\documentclass[11pt,a4paper]{article}
\usepackage[margin=2.45cm]{geometry}
\usepackage{amsmath,amssymb,amsthm,mathtools,bm}
\usepackage{booktabs,array}
\usepackage{hyperref}
\usepackage{microtype}
\usepackage[T1]{fontenc}
\usepackage{lmodern}
\hypersetup{colorlinks=true,linkcolor=black,citecolor=black,urlcolor=blue}
\allowdisplaybreaks[3]
\newtheorem{theorem}{Theorem}[section]
\newtheorem{lemma}[theorem]{Lemma}
\newtheorem{proposition}[theorem]{Proposition}
\newtheorem{problem}[theorem]{Problem}

\newcommand{\Tr}{\operatorname{Tr}}
\newcommand{\rank}{\operatorname{rank}}
\newcommand{\supp}{\operatorname{supp}}
\newcommand{\lin}{\mathrm{lin}}
\newcommand{\conc}{\mathrm{conc}}
\newcommand{\capc}{\mathrm{cap}}
\newcommand{\Mn}{M_n(\mathbb C)}

\title{\bfseries Sharp Concave-Function Transfer for Lee-Type Schatten Norm Inequalities}
\author{Xing Li\\[0.35em]
\small Center for Control Theory and Guidance Technology, Harbin Institute of Technology\\
\small Harbin 150001, China}
\date{}

\begin{document}
\maketitle

\begin{abstract}
We compare the sharp constants in Lee-type Schatten norm inequalities for finite families of complex matrices with the corresponding constants obtained after applying an arbitrary nonnegative concave function to the relevant absolute values. We prove that the nonlinear problem has exactly the same best constant as the underlying linear problem for every finite matrix dimension, every number of summands, and every Schatten exponent, including the operator-norm endpoint. For finite exponents, the proof proceeds through a lossless single-cap reduction, an exact positive-mixture representation for finite cap combinations, and a noncommutative reassembly argument based on a weighted Schatten contraction and the Araki--Lieb--Thirring inequality. Finite-spectrum interpolation and a perturbation argument at zero then yield the general concave case. Thus every sharp linear result transfers without loss to the full class of nonnegative concave functions.
\end{abstract}

\noindent\textbf{Keywords:} Lee-type inequalities; Schatten norms; nonnegative concave functions; cap functions; weighted noncommutative $L_p$; Araki--Lieb--Thirring; sharp constants.

\section{The problem and the main result}

If $A,B\ge0$ and $f:[0,\infty)\to[0,\infty)$ is concave, Bourin--Uchiyama~\cite{BU} proved that, for every unitarily invariant norm,
\[
 |||f(A+B)|||\le |||f(A)+f(B)|||.
\]
Bourin~\cite[Theorem~2.1]{Bourin2010} extended this phenomenon to normal matrices: if $A$ and $B$ are normal, then
\[
 |||f(|A+B|)|||\le |||f(|A|)+f(|B|)|||
\]
for every symmetric norm. He further observed that the same argument applies to arbitrary finite families of normal matrices. Thus, for normal summands, the concave-function problem already admits the constant $1$.

For unrestricted matrices the situation is different. Lee~\cite{Lee2010,Lee2012} subsequently studied the corresponding absolute-value form for general matrices and raised the problem of determining the best constants for Schatten norms. In the two-matrix Frobenius case, the best constant
\[
 \sqrt{\frac{1+\sqrt2}{2}}
\]
was proved by Lin--Zhang~\cite{LinZhang}; Zhang~\cite{Zhang2025} later gave another proof. Tang--Zhang~\cite{TangZhang} extended the linear problem to $m$ matrices and obtained the dimension-free sharp Frobenius constant
\[
 \sqrt{\frac{1+\sqrt m}{2}},
\]
while also studying general Schatten exponents.

Tang--Zhang~\cite[Proposition~3.2]{TangZhang} also proved that if $f$ is nonnegative, concave, and geometrically convex, then for $1\le p\le\infty$,
\[
 \left\|f\left(\left|\sum_{j=1}^mA_j\right|\right)\right\|_p
 \le m^{\frac12-\frac1{2p}}
 \left\|\sum_{j=1}^mf(|A_j|)\right\|_p.
\]
We remove the geometric-convexity assumption and determine the corresponding best constant.

For $1\le p<\infty$, fix $m,n\ge1$ and define
\begin{equation}\label{eq:clin}
 c^{\lin}_{p,m,n}
 :=\sup_{A_1,\dots,A_m\in\Mn}
 \frac{\left\|\sum_{j=1}^mA_j\right\|_p}
 {\left\|\sum_{j=1}^m|A_j|\right\|_p},
\end{equation}
where the denominator is required to be nonzero. Define further
\begin{equation}\label{eq:cconc}
 c^{\conc}_{p,m,n}
 :=\sup_{\substack{f:[0,\infty)\to[0,\infty)\ \text{concave}\\
 A_1,\dots,A_m\in\Mn}}
 \frac{\left\|f\left(\left|\sum_{j=1}^mA_j\right|\right)\right\|_p}
 {\left\|\sum_{j=1}^mf(|A_j|)\right\|_p},
\end{equation}
where cases with zero denominator are excluded from the supremum. The same definitions are used for $p=\infty$. Taking $f(t)=t$ gives
$c^{\conc}_{p,m,n}\ge c^{\lin}_{p,m,n}$.

\begin{theorem}\label{thm:main}
For every $m,n\ge1$ and $1\le p\le\infty$,
\begin{equation}\label{eq:main}
 c^{\conc}_{p,m,n}=c^{\lin}_{p,m,n}.
\end{equation}
Equivalently, for every nonnegative concave function $f$ and every $A_1,\dots,A_m\in\Mn$,
\begin{equation}\label{eq:mainineq}
 \left\|f\left(\left|\sum_{j=1}^mA_j\right|\right)\right\|_p
 \le c^{\lin}_{p,m,n}
 \left\|\sum_{j=1}^mf(|A_j|)\right\|_p.
\end{equation}
The constant is best possible.
\end{theorem}

Thus the full Lee problem for nonnegative concave functions is reduced to the corresponding linear best-constant problem. Recent results at the linear level show a marked difference between $1<p<2$ and $p\ge2$: Zeng--Liu--Ratnavelu~\cite{ZLR} gave an explicit $2\times2$ rank-one counterexample at $p=3/2$; Qiu~\cite{Qiu} proved that the Tang--Zhang two-matrix candidate formula fails for every $1<p<2$ and determined the sharp two-matrix constants for all $2\le p<\infty$; Pang~\cite{Pang2026} extended the failure of the candidate formula to every $m\ge2$ and every $1<p<2$. Teng Zhang~\cite{TengZhang2026} obtained several exact values and estimates for the fixed-dimensional linear constants. Combined with the present result, these works immediately yield the corresponding statements for the full class of nonnegative concave functions.

To the best of the author's knowledge, the identity \eqref{eq:main} for every fixed $p,m,n$ has not previously appeared in the literature.

The proof is organized around three ingredients. First, we show that passing from the identity function to a single cap is lossless. Second, positive finite combinations of caps admit an exact positive-mixture representation after taking the $p$th power. Third, a weighted Schatten contraction and the Araki--Lieb--Thirring inequality reassemble these scalar cap decompositions after the matrices are summed. The next section develops these ingredients, and Section~3 combines them with finite-spectrum interpolation and a perturbation argument at zero to prove Theorem~\ref{thm:main}; the operator-norm endpoint is then obtained by passage to the limit.

Throughout, $M_n(\mathbb C)$ denotes the complex $n\times n$ matrices, $|A|=(A^*A)^{1/2}$, and $\|A\|_p=(\Tr|A|^p)^{1/p}$ for $1\le p<\infty$, while $\|A\|_\infty$ is the operator norm. We write $\Tr$, $\rank$, and $\supp$ for the trace, rank, and support projection, respectively; spectral functions and spectral projections such as $\mathbf1_E(P)$ are understood through the usual functional calculus. For a linear map $\Phi$ between finite-dimensional matrix $C^*$-algebras, $\Phi_*$ denotes the trace adjoint, determined by $\Tr[\sigma\Phi(X)]=\Tr[\Phi_*(\sigma)X]$.

\section{Ingredients for the transfer argument}

\subsection{Single-cap transfer}

Let $h_r(t):=\min\{t,r\}$, where $r>0$. For $1\le p<\infty$ define the single-cap best constant
\[
 c^{\capc}_{p,m,n}(r)
 :=\sup_{A_1,\dots,A_m}
 \frac{\left\|h_r\left(\left|\sum_jA_j\right|\right)\right\|_p}
 {\left\|\sum_jh_r(|A_j|)\right\|_p}.
\]
We prove in this section that this constant equals the linear constant. The key ingredient is the following low-rank perturbation inequality.

\begin{lemma}\label{lem:lowrank}
Let $S,L\in\Mn$, $r>0$, $1\le p<\infty$, and suppose that $\rank(S-L)\le K$. Then
\begin{equation}\label{eq:lowrank}
 \|h_r(|S|)\|_p^p\le Kr^p+\|L\|_p^p.
\end{equation}
\end{lemma}

\begin{proof}
If $K\ge n$, then the left-hand side is at most $nr^p\le Kr^p$. Assume $K<n$. Write the singular values in decreasing order as $s_1(\cdot)\ge\cdots\ge s_n(\cdot)$. By the Ky Fan singular-value inequality (see, e.g., Bhatia~\cite[Chapter~III]{Bhatia}), applied to $S=L+(S-L)$,
\[
 s_{k+K}(S)\le s_k(L)+s_{K+1}(S-L)=s_k(L),
 \qquad 1\le k\le n-K,
\]
because $\rank(S-L)\le K$.
Hence
\[
 \sum_{i=1}^n\min\{s_i(S),r\}^p
 \le Kr^p+\sum_{k=1}^{n-K}s_k(L)^p
 \le Kr^p+\|L\|_p^p.
\]
\end{proof}

\begin{lemma}[McCarthy inequality]\label{lem:mccarthy}
If $X,Y\ge0$ and $p\ge1$, then
\begin{equation}\label{eq:mccarthy}
 \Tr(X+Y)^p\ge \Tr X^p+\Tr Y^p.
\end{equation}
See McCarthy~\cite{McCarthy}.
\end{lemma}

\begin{theorem}[Lossless transfer for a single cap]\label{thm:singlecap}
For every $1\le p<\infty$, $m,n\ge1$, and $r>0$,
\begin{equation}\label{eq:singlecap}
 c^{\capc}_{p,m,n}(r)=c^{\lin}_{p,m,n}.
\end{equation}
\end{theorem}

\begin{proof}
Set $c=c^{\lin}_{p,m,n}$ and write the polar decompositions $A_j=U_jP_j$, where $P_j=|A_j|$. Let
\[
 E_j=\mathbf1_{(r,\infty)}(P_j),\qquad
 L_j=P_j(I-E_j).
\]
Then
\[
 h_r(P_j)=rE_j+L_j.
\]
Put
\[
 S=\sum_jA_j,\qquad L=\sum_jU_jL_j,
 \qquad K=\sum_j\rank E_j.
\]
Since
\[
 S-L=\sum_jU_jP_jE_j,
\]
we have $\rank(S-L)\le K$. Lemma~\ref{lem:lowrank} gives
\begin{equation}\label{eq:cap1}
 \|h_r(|S|)\|_p^p\le Kr^p+\|L\|_p^p.
\end{equation}
The polar partial isometry satisfies $U_j^*U_j=\supp P_j$, while $\supp L_j\le\supp P_j$. Hence
\[
 |U_jL_j|^2=L_jU_j^*U_jL_j=L_j^2,
 \qquad\text{so}\qquad |U_jL_j|=L_j.
\]
By the definition of the linear best constant,
\[
 \|L\|_p\le c\left\|\sum_jL_j\right\|_p.
\]
On the other hand, let
\[
 D_r=\sum_jh_r(P_j)=r\sum_jE_j+\sum_jL_j.
\]
Applying \eqref{eq:mccarthy} first to the two positive summands in $D_r$, and then iteratively to $\sum_jE_j$, yields
\begin{align*}
 \|D_r\|_p^p
 &\ge r^p\Tr\left(\sum_jE_j\right)^p
       +\left\|\sum_jL_j\right\|_p^p\\
 &\ge r^p\sum_j\Tr E_j
       +\left\|\sum_jL_j\right\|_p^p\\
 &=Kr^p+\left\|\sum_jL_j\right\|_p^p.
\end{align*}
Since $c\ge1$, combining this with \eqref{eq:cap1} gives
\[
 \|h_r(|S|)\|_p^p
 \le c^p\|D_r\|_p^p.
\]
Therefore $c^{\capc}_{p,m,n}(r)\le c$.

The reverse inequality follows by scaling: given any family approaching the linear extremal ratio, multiply all matrices by a sufficiently small $\varepsilon>0$ so that all relevant singular values do not exceed $r$. On the relevant spectra, $h_r$ then agrees with the identity function, and hence $c^{\capc}_{p,m,n}(r)\ge c$.
\end{proof}

\subsection{Positive mixtures of caps}

\begin{lemma}\label{lem:pmix}
Fix $1\le p<\infty$. Suppose
\begin{equation}\label{eq:qcap}
 q(t)=\sum_{k=1}^Na_kh_{r_k}(t),\qquad a_k\ge0,
\end{equation}
and $q\not\equiv0$. Then there exists a finite positive Borel measure $\nu=\nu_{q,p}$ on $(0,\infty)$ such that, for every $t\ge0$,
\begin{equation}\label{eq:pmix}
 q(t)^p=\int_0^\infty h_u(t)^p\,d\nu(u).
\end{equation}
\end{lemma}

\begin{proof}
The function $q$ is nonnegative, nondecreasing, concave, satisfies $q(0)=0$, and is constant beyond the largest cap threshold. From \eqref{eq:qcap}, for $t>0$ we directly obtain
\[
 q'_+(t)=\sum_{k=1}^N a_k\mathbf1_{\{t<r_k\}}.
\]
Thus $q'_+$ is nonnegative, right-continuous, nonincreasing, and eventually zero. Moreover, concavity together with $q(0)=0$ implies that $q(t)/t$ is nonincreasing. Hence
\[
 q'_+(t)\left(\frac{q(t)}{t}\right)^{p-1}
\]
is nonnegative, right-continuous, nonincreasing, and eventually zero, and it has a finite limit at $0+$. Thus it is bounded, vanishes at infinity, and is the tail function of a unique finite positive Lebesgue--Stieltjes measure $\nu$ satisfying
\begin{equation}\label{eq:tail}
 \nu((t,\infty))
 =q'_+(t)\left(\frac{q(t)}{t}\right)^{p-1},
 \qquad t>0.
\end{equation}

Using
\[
 h_u(t)^p
 =p\int_0^t s^{p-1}\mathbf1_{\{s<u\}}\,ds
\]
and Tonelli's theorem, we obtain
\begin{align*}
 \int h_u(t)^p\,d\nu(u)
 &=p\int_0^ts^{p-1}\nu((s,\infty))\,ds\\
 &=p\int_0^ts^{p-1}q'_+(s)
      \left(\frac{q(s)}s\right)^{p-1}ds\\
 &=p\int_0^tq(s)^{p-1}q'_+(s)\,ds
 =q(t)^p.
\end{align*}
The last step uses the absolute continuity of $q$ and $q(0)=0$.
\end{proof}

\subsection{Noncommutative reassembly}

The main obstacle is that the scalar cap decomposition has to survive a noncommutative sum. At the heart of the reassembly step is the following finite-dimensional weighted Schatten contraction. Its interpolation mechanism is the finite-dimensional tracial counterpart of the positive-map extension theorem of Haagerup--Junge--Xu~\cite[Theorem~5.1]{HJX}; we include the elementary proof needed here for completeness.

\begin{lemma}[Weighted Schatten $L_p$ contraction]\label{lem:weightedLp}
Let $\Phi:\mathcal A\to\mathcal B$ be a unital positive map, let $\sigma\in\mathcal B$ be strictly positive, and set
\[
 \tau:=\Phi_*(\sigma).
\]
Assume that $\tau$ is strictly positive. Then, for every $1\le p\le\infty$ and every $X\in\mathcal A$,
\begin{equation}\label{eq:weightedLp}
 \left\|\sigma^{1/(2p)}\Phi(X)\sigma^{1/(2p)}\right\|_p
 \le
 \left\|\tau^{1/(2p)}X\tau^{1/(2p)}\right\|_p.
\end{equation}
For $p=\infty$, we use the convention $1/p=0$.
\end{lemma}

\begin{proof}
The case $p=\infty$ is the usual operator-norm contraction of a unital positive map. For $1\le p<\infty$, set
\begin{equation}\label{eq:Tz}
 T_z(Y):=\sigma^{z/2}\Phi\!\left(\tau^{-z/2}Y\tau^{-z/2}\right)\sigma^{z/2},
 \qquad 0\le\Re z\le1.
\end{equation}
For $z=it$, the imaginary powers $\sigma^{it/2}$ and $\tau^{-it/2}$ are unitary. Hence unitary invariance together with the operator-norm contractivity of a unital positive map (Russo--Dye) gives
\begin{equation}\label{eq:endinf}
 \|T_{it}(Y)\|_\infty\le\|Y\|_\infty.
\end{equation}
Also $K:=T_1$ is positive and trace-preserving, since by the definition of the trace adjoint,
\[
 \Tr K(Y)
 =\Tr\!\left[\sigma\,\Phi(\tau^{-1/2}Y\tau^{-1/2})\right]
 =\Tr\!\left[\tau\,\tau^{-1/2}Y\tau^{-1/2}\right]
 =\Tr Y.
\]
Hence $K_*$ is unital positive and, by duality and Russo--Dye, $\|K\|_{1\to1}\le1$. Since
\[
 T_{1+it}(Y)=\sigma^{it/2}K\!\left(\tau^{-it/2}Y\tau^{-it/2}\right)\sigma^{it/2},
\]
we have
\begin{equation}\label{eq:end1}
 \|T_{1+it}(Y)\|_1\le\|Y\|_1.
\end{equation}
The family $T_z$ is analytic on the strip and uniformly bounded on its closure in finite dimension. Interpolating between the $L_\infty$ and $L_1$ boundary estimates by Stein--Riesz--Thorin therefore yields
\[
 \|T_{1/p}(Y)\|_p\le\|Y\|_p.
\]
This is exactly the finite-dimensional interpolation mechanism relevant to Haagerup--Junge--Xu~\cite[Theorem~5.1]{HJX}. Substituting $Y=\tau^{1/(2p)}X\tau^{1/(2p)}$ gives \eqref{eq:weightedLp}.
\end{proof}

With the weighted contraction in hand, we now reassemble the scalar cap decomposition. The identity \eqref{eq:pmix} applies separately to each $P_j$, whereas the Schatten norm involves the noncommutative sum $\sum_j h_u(P_j)$. The required estimate is the following.

\begin{proposition}\label{prop:reassembly}
Let $1\le p<\infty$, let $q$ be as in \eqref{eq:qcap}, and let $\nu=\nu_{q,p}$ be the positive measure from Lemma~\ref{lem:pmix}. For any positive semidefinite matrices $P_1,\dots,P_m$,
\begin{equation}\label{eq:reassembly}
 \int_0^\infty
 \left\|\sum_{j=1}^m h_u(P_j)\right\|_p^p\,d\nu(u)
 \le
 \left\|\sum_{j=1}^m q(P_j)\right\|_p^p.
\end{equation}
\end{proposition}

\begin{proof}
If $q\equiv0$, the claim is immediate. Hence assume $q\not\equiv0$. Set
\[
 X_j=q(P_j),\qquad G=\sum_{j=1}^mX_j,
\]
and
\[
 H_{j,u}=h_u(P_j),\qquad D_u=\sum_{j=1}^mH_{j,u}.
\]
Since $q(t)>0$ for every $t>0$,
\[
 \supp X_j=\supp P_j.
\]
Let $\mathcal H=\supp G$ and $\mathcal H_j=\supp X_j$, and restrict to these support spaces. Then $G>0$ on $\mathcal H$ and $X_j>0$ on $\mathcal H_j$; moreover $H_{j,u}$ acts on $\mathcal H_j$ and $D_u$ on $\mathcal H$.

Define
\[
 V_j:=X_j^{1/2}G^{-1/2}:\mathcal H\to\mathcal H_j.
\]
Then
\[
 \sum_jV_j^*V_j
 =G^{-1/2}\left(\sum_jX_j\right)G^{-1/2}
 =I_{\mathcal H}.
\]
Therefore
\[
 \Phi:\bigoplus_{j=1}^mB(\mathcal H_j)\to B(\mathcal H),
 \qquad
 \Phi((Y_j)_j)=\sum_jV_j^*Y_jV_j
\]
is a UCP map.

Let
\[
 R_{j,u}:=X_j^{-1/2}H_{j,u}X_j^{-1/2}.
\]
Because $X_j$ and $H_{j,u}$ are both functions of the same matrix $P_j$, they commute, and
\begin{equation}\label{eq:PhiR}
 \Phi((R_{j,u})_j)=G^{-1/2}D_uG^{-1/2}.
\end{equation}
This is the only commutativity used at this stage: no commutativity is assumed between different $P_j$'s, nor between any $X_j$ and their sum $G$.

In Lemma~\ref{lem:weightedLp}, take
\[
 \sigma=G^p.
\]
The $j$th direct-sum block of the trace-adjoint weight $\tau=\Phi_*(G^p)$ is
\begin{equation}\label{eq:tauj}
 \tau_j=V_jG^pV_j^*
 =X_j^{1/2}G^{p-1}X_j^{1/2},
\end{equation}
and it is strictly positive on $\mathcal H_j$: indeed, $V_j:\mathcal H\to\mathcal H_j$ is onto, so $V_j^*$ is injective, while $G^p>0$ on $\mathcal H$. Since $\sigma^{1/(2p)}=G^{1/2}$, identity \eqref{eq:PhiR} shows that the left-hand side of the weighted contraction is exactly $\|D_u\|_p$. Using the direct-sum form of $\tau$ and $(R_{j,u})_j$ on the right-hand side, we obtain
\begin{equation}\label{eq:weightedapply}
 \|D_u\|_p^p
 \le
 \sum_j
 \Tr\left(
 \tau_j^{1/(2p)}R_{j,u}\tau_j^{1/(2p)}
 \right)^p.
\end{equation}

In general $\tau_j$ and $R_{j,u}$ do not commute. This is precisely the genuinely noncommutative point of the argument. For $A,B\ge0$ and $p\ge1$, the Araki--Lieb--Thirring inequality~\cite{Araki} gives
\begin{equation}\label{eq:ALTform}
 \Tr\left(A^{1/(2p)}BA^{1/(2p)}\right)^p
 \le \Tr(AB^p).
\end{equation}
Applying this to each summand in \eqref{eq:weightedapply} removes the noncommutative sandwich without imposing any commutativity assumption and yields
\begin{equation}\label{eq:ALTapply}
 \|D_u\|_p^p
 \le\sum_j\Tr(\tau_jR_{j,u}^p).
\end{equation}

We now return to the local commutativity inside a fixed $j$. Since $X_j$ and $H_{j,u}$ commute,
\[
 R_{j,u}^p=X_j^{-p}H_{j,u}^p.
\]
The spectral-calculus version of Lemma~\ref{lem:pmix} gives
\[
 \int H_{j,u}^p\,d\nu(u)=q(P_j)^p=X_j^p,
\]
so on $\mathcal H_j$,
\begin{equation}\label{eq:Rintegral}
 \int R_{j,u}^p\,d\nu(u)=I_{\mathcal H_j}.
\end{equation}
Integrating \eqref{eq:ALTapply} and using \eqref{eq:Rintegral},
\begin{align*}
 \int\|D_u\|_p^p\,d\nu(u)
 &\le\sum_j\Tr\tau_j\\
 &=\sum_j\Tr\left(G^{p-1}X_j\right)\\
 &=\Tr G^p
 =\left\|\sum_jq(P_j)\right\|_p^p.
\end{align*}
This is \eqref{eq:reassembly}.
\end{proof}

The mechanism of Proposition~\ref{prop:reassembly} is worth emphasizing. Commutativity is used only for functions of a single $P_j$. All interactions among different matrices are absorbed by the UCP map $\Phi$; the choice $\sigma=G^p$ recovers the unweighted sum $D_u$, and the Araki--Lieb--Thirring inequality is exactly what controls the remaining noncommutative sandwich. This replaces the quadratic-form reassembly available at $p=2$ and is what allows the argument to work for general finite Schatten exponents.

\section{Proof of the sharp concave-function transfer}

We now combine the preceding ingredients and pass from finite cap mixtures to arbitrary nonnegative concave functions.

Every nonnegative concave function $f:[0,\infty)\to[0,\infty)$ is nondecreasing. We first treat the case $f(0)=0$.

\begin{lemma}\label{lem:finiteinterp}
Let $f:[0,\infty)\to[0,\infty)$ be concave and let $0<t_1<\cdots<t_N$. Then there is a positive finite combination of caps $q$ with $q(0)=0$ and $q(t_i)=f(t_i)$ for every $i$.
\end{lemma}

\begin{proof}
Set
\[
 d_1:=\frac{f(t_1)}{t_1},\qquad
 d_k:=\frac{f(t_k)-f(t_{k-1})}{t_k-t_{k-1}}\quad(2\le k\le N).
\]
Since $f\ge0$ is concave,
\[
 d_1\ge \frac{f(t_1)-f(0)}{t_1}\ge d_2\ge\cdots\ge d_N\ge0.
\]
For any $R>t_N$,
\begin{equation}\label{eq:interpformula}
 q(t):=d_Nh_R(t)+\sum_{k=1}^{N-1}(d_k-d_{k+1})h_{t_k}(t)
\end{equation}
has nonnegative coefficients and satisfies $q(0)=0$ and $q(t_i)=f(t_i)$.
\end{proof}

\begin{proposition}\label{prop:f0finite}
Let $1\le p<\infty$, and let $f:[0,\infty)\to[0,\infty)$ be concave with $f(0)=0$. Then
\begin{equation}\label{eq:f0finite}
 \left\|f\left(\left|\sum_jA_j\right|\right)\right\|_p
 \le c^{\lin}_{p,m,n}
 \left\|\sum_jf(|A_j|)\right\|_p.
\end{equation}
\end{proposition}

\begin{proof}
Set
\[
 C=\left|\sum_jA_j\right|,\qquad P_j=|A_j|,
 \qquad c=c^{\lin}_{p,m,n}.
\]
Apply Lemma~\ref{lem:finiteinterp} to the positive eigenvalues of $C,P_1,\dots,P_m$. Since $f(0)=q(0)=0$, spectral calculus gives $q(C)=f(C)$ and $q(P_j)=f(P_j)$. If $q\equiv0$, the claim is immediate. Otherwise let $\nu$ be the measure from Lemma~\ref{lem:pmix}. Theorem~\ref{thm:singlecap} and Proposition~\ref{prop:reassembly} give
\begin{align*}
 \|f(C)\|_p^p
 &=\|q(C)\|_p^p
 =\int\|h_u(C)\|_p^p\,d\nu(u)\\
 &\le c^p\int
 \left\|\sum_jh_u(P_j)\right\|_p^p\,d\nu(u)\\
 &\le c^p\left\|\sum_jq(P_j)\right\|_p^p
 =c^p\left\|\sum_jf(P_j)\right\|_p^p.
\end{align*}
\end{proof}

\begin{lemma}\label{lem:f0positive}
Proposition~\ref{prop:f0finite} remains valid for every nonnegative concave function $f$.
\end{lemma}

\begin{proof}
Fix $A_1,\dots,A_m$, and write $C=|\sum_jA_j|$ and $P_j=|A_j|$. If all these matrices are zero, the claim is immediate. Let $\lambda_*>0$ be the smallest positive eigenvalue among $C,P_1,\dots,P_m$, and choose $0<\delta<\lambda_*$. On $[0,\delta]$, replace $f$ by the line segment joining $(0,f(0))$ and $(\delta,f(\delta))$, while leaving it unchanged on $[\delta,\infty)$; denote the resulting function by $\widetilde f$. Since the secant slopes of a concave function are nonincreasing, $\widetilde f$ is still nonnegative, concave, nondecreasing, and continuous at $0$. Moreover,
\begin{equation}\label{eq:ftildeequal}
 \widetilde f(C)=f(C),\qquad \widetilde f(P_j)=f(P_j).
\end{equation}

The polynomial
\[
 \left(\prod_{j=1}^m\det A_j\right)
 \det\left(\sum_{j=1}^mA_j\right)
\]
on $(\Mn)^m$ is not identically zero. The zero set of a nonzero complex polynomial has empty interior, hence its nonzero set is dense. Choose $A_j^{(\varepsilon)}\to A_j$ so that every $A_j^{(\varepsilon)}$ and $\sum_jA_j^{(\varepsilon)}$ are invertible. Lemma~\ref{lem:finiteinterp}, applied to all eigenvalues of the corresponding absolute-value matrices, gives a positive cap combination $q_\varepsilon$ with $q_\varepsilon(0)=0$ that agrees there with $\widetilde f$. Applying Proposition~\ref{prop:f0finite} to $q_\varepsilon$ and then using these spectral identities yields
\[
 \left\|\widetilde f\left(\left|\sum_jA_j^{(\varepsilon)}\right|\right)\right\|_p
 \le c^{\lin}_{p,m,n}
 \left\|\sum_j\widetilde f(|A_j^{(\varepsilon)}|)\right\|_p.
\]
Letting $\varepsilon\to0$, continuity of $\widetilde f$ and finite-dimensional functional calculus, together with \eqref{eq:ftildeequal}, gives the result.
\end{proof}

\begin{proof}[Proof of Theorem~\ref{thm:main}]
First let $1\le p<\infty$. Proposition~\ref{prop:f0finite} and Lemma~\ref{lem:f0positive} give
$c^{\conc}_{p,m,n}\le c^{\lin}_{p,m,n}$, while the reverse inequality follows by taking $f(t)=t$.

Now let $p=\infty$. In finite dimensions, Schatten norms satisfy
\[
 \|X\|_\infty\le\|X\|_p\le n^{1/p}\|X\|_\infty.
\]
Therefore
\begin{equation}\label{eq:clinftycompare}
 n^{-1/p}c^{\lin}_{\infty,m,n}
 \le c^{\lin}_{p,m,n}
 \le n^{1/p}c^{\lin}_{\infty,m,n},
\end{equation}
so $c^{\lin}_{p,m,n}\to c^{\lin}_{\infty,m,n}$. For fixed $f,A_1,\dots,A_m$, the finite-$p$ case gives
\[
 \left\|f\left(\left|\sum_jA_j\right|\right)\right\|_p
 \le c^{\lin}_{p,m,n}
 \left\|\sum_jf(|A_j|)\right\|_p.
\]
Letting $p\to\infty$ gives the operator-norm case of \eqref{eq:mainineq}. Sharpness again follows from $f(t)=t$.
\end{proof}

\section{Consequences and the remaining problem}

The main theorem and the existing linear results yield the following consequences.

\begin{proposition}\label{prop:applications}
For every nonnegative concave function $f$, the following hold:
\begin{enumerate}
\item For $p=1$,
\[
 \left\|f\left(\left|\sum_{j=1}^mA_j\right|\right)\right\|_1
 \le
 \left\|\sum_{j=1}^mf(|A_j|)\right\|_1,
\]
and the constant $1$ is best possible.

\item For $p=2$,
\[
 c^{\conc}_{2,m,n}
 =
 \sqrt{\frac{1+\sqrt{\min\{m,n\}}}{2}}.
\]

\item For $p=\infty$,
\[
 c^{\conc}_{\infty,m,n}=\sqrt{\min\{m,n\}}.
\]

\item As a Frobenius constant uniform over all dimensions,
\[
 \sup_{n\ge1}c^{\conc}_{2,m,n}
 =
 \sqrt{\frac{1+\sqrt m}{2}}.
\]

\item For $m=2$ and $2\le p<\infty$, let $x_p>1$ satisfy
\[
 x_p^p=2x_p+1.
\]
Then
\[
 \sup_{n\ge1}c^{\conc}_{p,2,n}
 =
 C_p:=
 \frac{\sqrt{x_p(x_p+1)}}{(x_p^p+1)^{1/p}}.
\]
\end{enumerate}
\end{proposition}

\begin{proof}
The first statement follows from the Schatten $1$ triangle inequality, which gives $c^{\lin}_{1,m,n}=1$. The second and third statements follow from Teng Zhang's fixed-dimensional linear results~\cite{TengZhang2026}; the fourth also follows by combining the main theorem with the dimension-free sharp Frobenius constant of Tang--Zhang~\cite{TangZhang}; and the fifth follows from Qiu's sharp two-matrix result~\cite{Qiu}. Applying Theorem~\ref{thm:main} gives each assertion.
\end{proof}

The candidate formula proposed by Tang--Zhang~\cite{TangZhang} for $1<p<2$ is now known to fail throughout the entire interval. Zeng--Liu--Ratnavelu~\cite{ZLR} gave an explicit low-dimensional counterexample at $p=3/2$, and Qiu~\cite{Qiu} constructed counterexamples for two matrices for every $1<p<2$. Pang~\cite{Pang2026} further proved that the candidate formula fails for every $m\ge2$ and every $1<p<2$. The correct sharp linear constants in this range remain to be determined.

\begin{problem}
For $1<p<2$, determine the fixed-dimensional best constant $c^{\lin}_{p,m,n}$ or the all-dimensional constant
\[
 c_p^{\lin}(m):=\sup_{n\ge1}c^{\lin}_{p,m,n},
\]
and study the extremal cases.
\end{problem}

\section*{Acknowledgments}
The author thanks Jean-Christophe Bourin for helpful correspondence and for drawing attention to his earlier normal-matrix subadditivity result.

OpenAI's GPT-5.6 Sol assisted in developing a key construction in the noncommutative reassembly step, as well as in language editing, exposition, and consistency checks of the manuscript. This construction made it possible to extend an earlier quadratic-form argument, which was valid only for $p=2$, to the full range $p>1$. The author independently verified the mathematical arguments, references, and final text and takes full responsibility for the content of the paper.

\end{document}